\documentclass[10pt]{amsart}

\usepackage{amsfonts,amssymb,amsmath,amscd,amstext}
\usepackage[colorlinks=true,linkcolor=blue,citecolor=blue, backref=page]{hyperref}
\usepackage[utf8]{inputenc}
\usepackage{microtype}
\usepackage{graphicx}
\usepackage{changes}
\usepackage{comment}

\renewcommand{\epsilon}{\eps}

\newcommand{\N}{{\mathbb N}}
\newcommand{\R}{{\mathbb R}}
\newcommand{\rr}{{\mathbb R}}
\newcommand{\X}{{\mathbb X}}

\newcommand{\eps}{\varepsilon}

\def \LL {\mathcal{L}_{\alpha}}
\def \de {\partial}

\newcommand{\pnorm}[2][]{\if #1'' \left|#2\right|_p \else \left|#2\right|_{#1} \fi}

\newcommand{\escpr}[1]{\langle#1\rangle}

\def \LL {\mathcal{L}_{\alpha}}
\newtheorem*{axiom}{Assumptions on $f$}
\renewcommand{\theta}{\vartheta}
\allowdisplaybreaks

\renewcommand*{\backref}[1]{}
\renewcommand*{\backrefalt}[4]{
    \ifcase #1 (Not cited.)
    \or        (Cited on page~#2.)
    \else      (Cited on pages~#2.)
    \fi}

\newtheorem{theorem}{Theorem}[section]
\newtheorem{proposition}[theorem]{Proposition}
\newtheorem{lemma}[theorem]{Lemma}

\theoremstyle{definition}

\newtheorem{remark}[theorem]{Remark}

\newtheorem{definition}[theorem]{Definition} 

\theoremstyle{remark}

\numberwithin{equation}{section}

\makeatletter
\@namedef{subjclassname@2020}{\textup{2020} Mathematics Subject Classification}
\makeatother

\begin{document}

\title[A local-nonlocal resonance problem ]{An Ahmad-Lazer-Paul-type result for indefinite mixed local-nonlocal problems}

\author[G.~Giovannardi]{Gianmarco Giovannardi}
\address{Dipartimento di Matematica Informatica \newline\indent"U. Dini", Università degli Studi di Firenze,\newline\indent Viale Morgani 67/A, 50134, Firenze, Italy}
\email{gianmarco.giovannardi@unifi.it}

 \author[D.~Mugnai]{Dimitri Mugnai}
 \address{Dipartimento di Ecologia e Biologia (DEB)
 \newline\indent Università della Tuscia \newline\indent
 Largo dell'Università, 01100 Viterbo, Italy}
 \email{dimitri.mugnai@unitus.it}

\author[E.~Vecchi]{Eugenio Vecchi}
\address{Dipartimento di Matematica
\newline\indent Università di Bologna
\newline\indent Piazza di Porta S. Donato 5, 40126 Bologna, Italy
}
\email{eugenio.vecchi2@unibo.it}

\date{\today}



\thanks{G. Giovannardi is supported by INdAM-GNAMPA Project ``Analisi geometrica in strutture subriemanniane''. D. Mugnai is supported by the FFABR ``Fondo per il finanziamento delle attivit\`a base di ricerca'' 2017 and by the INdAM-GNAMPA Project ``PDE ellittiche a diffusione mista''.
E. Vecchi is supported by INdAM-GNAMPA Project ``PDE ellittiche a diffusione mista''.}

\bibliographystyle{abbrv}

\begin{abstract}
 We prove the existence and multiplicity of weak solutions for a mixed local-nonlocal problem at resonance. In particular, we consider a not necessarily positive operator which appears in models describing the propagation of flames. A careful adaptation of well known variational methods is required to deal with the possible existence of negative eigenvalues.
\end{abstract}

\maketitle

\thispagestyle{empty}

\section{Introduction}
Let $\Omega \subset \mathbb{R}^n$ be an open and bounded set with $C^1$-smooth boundary $\partial \Omega$ with $n>2$. We consider the following Dirichlet boundary value problem
\begin{equation}\label{eq:Problem}
\left\{ \begin{array}{rl}
\LL u = \lambda u + f(x,u), & \textrm{in } \Omega\\
u = 0, &\textrm{in } \mathbb{R}^n \setminus \Omega,
\end{array}\right.
\end{equation}
where 
\[
\LL u := -\Delta u +\alpha (-\Delta)^s u\,.
\]
Here $\alpha \in \mathbb{R}$ with no a priori restrictions, $\Delta u$ denotes the classical Laplace operator while $(-\Delta)^s u$, for fixed $s\in (0,1)$ and up to a multiplicative positive constant, is the fractional Laplacian, usually defined as
\[
(-\Delta)^s u(x) := C(n,s) \,\mbox{P.V.}\int_{\R^n}\frac{u(x)-u(y)}{|x-y|^{n+2s}}\, dy\,,
\]
where P.V. denotes the Cauchy principal value, that is
\[
\mbox{P.V.}\int_{\R^n}
  \frac{u(x)-u(y)}{|x-y|^{n+2s}}\, dy\\
  =\lim_{\epsilon\to0}\int_{\{y\in \R^n\,:\,|y-x|\geq \epsilon\}}
  \frac{u(x)-u(y)}{|x-y|^{n+2s}}\, dy,
\]
see \cite{LIBRO} for more details. 
Clearly, when $\alpha=0$, one recovers the classical Laplacian.
Finally, $\lambda \in \mathbb{R}$ is a variational Dirichlet eigenvalue of $\LL$ (hence \eqref{eq:Problem} is a problem at resonance), namely there are nontrivial solutions for the problem
\[
\left\{ \begin{array}{rl}
  \LL u = \lambda u, & \textrm{in } \Omega,\\
  u= 0, & \textrm{in } \mathbb{R}^n \setminus \Omega,
\end{array}\right.
\]
see Section \ref{sec.NotPrel} for the precise setting.

We suppose that $f$ satisfies the following assumptions.
\begin{axiom}
\item[$(f_{bc})$] $f:\overline{\Omega}\times \mathbb{R} \to \mathbb{R}$ is a bounded and  Carathéodory function, namely: 
\begin{enumerate}
\item $f(x, \cdot)$ is continuous in $\rr$ for a.e. $x \in \Omega$
\item $f(\cdot,t)$ is measurable in $\Omega$ for all $t \in \rr$.
\end{enumerate}
\item[$(F_{\pm\infty})$]
$$ \lim_{u\in Ker(\LL-\lambda),\atop \|u\|\to \infty}\int_\Omega F(x,u)dx=\pm\infty $$
uniformly in $x \in \Omega$.
\end{axiom}

\begin{remark}
The assumption $(F_{\pm \infty})$ is the so--called Ahmad-Lazer-Paul condition introduced in \cite{ALP} and often used in resonant problems, for instance see \cite{FoGa}.
A sufficient condition implying it is given by
\[
\displaystyle F(x,t)= \int_0^t f(x,\tau) d\tau \to \pm\infty \quad \quad \textrm{as } |t| \to +\infty,
\]
as can be quite easily checked, see e.g. \cite[Lemma 3.4]{FiSeVaZAMP}.
\end{remark} 
The goal of this paper is to extend to a mixed operator an existence and multiplicity result established in \cite{Rabinowitz78} for the Laplace operator. We state immediately our first result:
\begin{theorem}
\label{th:main}
Let $f$ satisfy $(f_{bc})$ and $(F_{\pm\infty})$ and suppose that $\lambda \in \rr$ is a variational Dirichlet eigenvalue of $\LL$. 
Then, the problem \eqref{eq:Problem} admits a weak solution $u \in \mathbb{X}(\Omega)$. 
\end{theorem}

We immediately clarify that Theorem \ref{th:main}: it is somehow {\it easy to get} for $\alpha >0$ or when
\begin{equation*}
-\dfrac{1}{C_s}<\alpha<0\,,
\end{equation*}
where $C_s>0$ is the best constant of the continuous embedding $H^{1}_{0} \subset H^{s}$ (see e.g. \cite{DRV}), i.e.
\begin{equation*}
[u]_s^2 := \iint_{\mathbb{R}^{2n}}\dfrac{|u(x)-u(y)|^2}{|x-y|^{n+2s}}\, dx dy \leq C_s \, \int_{\Omega}|\nabla u|^2 \, dx\,.
\end{equation*}
In this perspective, the probably more interesting case is  when $\alpha<-\dfrac{1}{C_s}$ is assumed. 
Indeed, the situation becomes suddenly more delicate, mainly because the local--nonlocal operator is not more positive definite, {\it indefinite operators}. As a consequence, the bilinear form naturally associated to it does not induce a scalar product nor a norm, the variational spectrum may exhibit negative eigenvalues and even the maximum principles may fail, see e.g. \cite{BDVV}. 
As a consequence, some inequalities cannot be adapted to our situations, since dividing by eigenvalues may reverse the sign, nullifying a verbatim adaptation of \cite{Rabinowitz78}. For this reason, we need to establish some crucial estimates for eigenspaces, see Lemmas \ref{lemma_claim} and \ref{lm:C1upper}.

\medskip

We stress that operators of this latter form (e.g. with $\alpha=-1$) do not have only a purely theoretical mathematical interest, indeed the play a role in applied sciences like combustion theory. We limit ourselves to mention that the stationary part in the original model proposed by Sivashinsky \cite{Siva} to deal with the instability of the propagation front of flames can be reduced to operator of the form previously described, and this may happen under phisically motivated assumptions, see e.g. \cite{Ibdah} and the references therein.

\medskip

We stress that Theorem \ref{th:main} is actually a further generalization of a result by Ahmad, Lazer and Paul \cite{ALP} where the authors dealt with a local operator at resonance. Despite well known, we want to recall here that the original proof in \cite{ALP} is by no means of variational flavour, and consists in a delicate construction of the solution using a sort of Galerkin method. As a matter of fact, a striking and not foreseeable consequence of \cite{ALP} is an existence result with proof relying on  the Saddle Point Theorem by Rabinowitz \cite{Rabinowitz78}, which is one of the cornerstones of variational methods in nonlinear analysis. We will follow this latter approach, which has already been used also in the pure nonlocal case in \cite{FiSeVaZAMP}. We notice that, as in \cite{FiSeVaZAMP}, it is possible to work in a slightly more general case considering weighted Dirichlet eigenvalues, where the weight $a$ is a Lipschitz function. However, treating solutions of $\LL=\lambda a(x)u$ doesn't change the spirit of our result, and for this reason we concentrate on eigenvalues without weight.

\medskip

In the spirit of \cite{Rabinowitz78}, under assumption $(F_{-\infty})$, we can also prove a multiplicity result under a few extra assumptions. The precise statement is the following
\begin{theorem}\label{th.main2}
Let $f$ satisfy $(f_{bc})$ and $(F_{-\infty})$ and suppose that $\lambda  \in \rr$ is a variational Dirichlet eigenvalue of $\LL$. 
We further assume that $f(x,0)=0$ and that $f(x,t)$ is odd in the $t$ variable. Finally, we assume that
\begin{equation*}
\tag{$F_{pos}$} \textrm{ there exists } r>0 \textrm{ such that } F(x,t)>0, \quad \textrm{ for } 0<|t|<r \textrm{ and } x\in \overline{\Omega}.
\end{equation*}
Then, problem \eqref{eq:Problem} admits at least $\dim(H^{0}_\lambda)$ distinct pair of nontrivial weak solutions, where $H^{0}_\lambda$ denotes the eigenspace associated to $\lambda$.
\end{theorem}

\begin{remark}
We remark that, as in \cite{Rabinowitz78}, we are able to prove the multiplicity result only when $(F_{-\infty})$, since only in such a case we are able to prove a decomposition of the space $\X(\Omega)$ for which the abstract multiplicity theorem \ref{teosotto} holds. Hence, the multiplicity result when $(F_{+\infty})$ holds is an open problem.
\end{remark}
\medskip

We close the introduction with a quite short overview on the more recent (elliptic) PDEs oriented literature. Problems driven by operators of mixed type, even with a nonsingular nonlocal operator \cite{DPFR}, have raised a certain interest in the last few years, for example in connection with the study of optimal animal foraging strategies (see e.g. \cite{DV} and the references therein). From the pure mathematical point of view, the superposition of such operators generates a lack of scale invariance which may lead to unexpected complications.

At the present stage, and without aim of completeness, the investigations have taken into consideration interior regularity and maximum principles (see e.g. \cite{BDVV,CDV22,DeFMin,GarainKinnunen,GarainLindgren}), boundary Harnack principle \cite{CKSV}, boundary regularity and overdetermined problems \cite{BMS, SUPR}, existence of solutions (see e.g. \cite{BDVV5, BMV, BMV2, DPLV1, DPLV2, GarainUkhlov, MPL, SalortVecchi, ArRa, Garain}) and shape optimization problems \cite{BDVV2,BDVV3}.

\medskip

The paper is organized as follows. In Section \ref{sec.NotPrel} we introduce some preliminary definitions and results, such as the Hilbert space $ \mathbb{X}(\Omega)$, the notion of weak solution of \eqref{eq:Problem} (as critical point of the functional $\mathcal{J}_{\lambda}$), the variational eigenvalue problem for  $\LL$ and  the crucial main lemmas. Section \ref{sc:proofs} is dedicated to the  proofs of the Theorem \ref{th:main} and Theorem \ref{th.main2};  we first deal with the geometry of the functional $\mathcal{J}_{\lambda}$ and the Palais-Smale condition, then we verify the hypothesis of the Saddle Point Theorem and   \cite[Theorem 1.9]{Rabinowitz78}. In the Appendix we recall the notion of Krasnoselskii genus and we state \cite[Theorem 1.9]{Rabinowitz78}.

\section{Assumptions, notation and preliminary results} \label{sec.NotPrel}

Let $\Omega\subseteq\R^n$ be a connected and bounded open set with $C^1$-smooth boundary $\partial\Omega$.
We define the space of solutions of problem \eqref{eq:Problem} as
\begin{equation*}
\mathbb{X}(\Omega) := \big\{u\in H^{1}(\R^n):\,
\text{$u\equiv 0$ a.e.\,on $\R^n\setminus\Omega$}\big\}.
\end{equation*}
Thanks to the regularity assumption on $\de\Omega$ (see \cite[Proposition 9.18]{Brezis}), we can identify the space $\mathbb{X}(\Omega)$ with the space $H_0^{1}(\Omega)$ in the following sense:
\begin{equation} \label{eq.identifXWzero}
    u \in H_0^{1}(\Omega)\,\,\Longleftrightarrow\,\,
    u\cdot\mathbf{1}_{\Omega}\in\mathbb{X}(\Omega)\,,
\end{equation}
where $\mathbf{1}_\Omega$ is the indicator function of $\Omega$.
From now on, we shall always identify a function $u\in H_0^{1}(\Omega)$ with $\hat{u} := u\cdot\mathbf{1}_\Omega\in\mathbb{X}(\Omega)$.

By the Poincar\'e inequality and \eqref{eq.identifXWzero}, we get that the quantity
\[
\|u\|_{\mathbb{X}} :=\left( \int_{\Omega}|\nabla u|^2\, dx\right)^{1/2},\quad u\in\mathbb{X}(\Omega)\,,\]
endows $\mathbb{X}(\Omega)$ with a structure of (real) Hilbert space, which is isometric to $H_0^{1}(\Omega)$.
To fix the notation, we denote by $\langle \cdot,\cdot \rangle_{\mathbb{X}}$ the scalar product which induces the norm above on $\mathbb{X}(\Omega)$.
We briefly recall that the space $\mathbb{X}(\Omega)$ is separable and reflexive, $C_0^\infty(\Omega)$ is dense in $\mathbb{X}(\Omega)$ and eventually that $\mathbb{X}(\Omega)$ compactly embeds in $L^{p}(\Omega)$ for any $p\in\left[1, \frac{2n}{n-2}\right)$ and in  
\[
H^s_0(\Omega):=\left\{H^{s}(\R^n):\,\text{$u\equiv 0$ a.e.\,on $\R^n\setminus\Omega$}\right\}
\]
by \cite[Theorem 16.1]{lions}.

With the correct functional setting, we are ready to give the suitable notion of weak solution for problem \eqref{eq:Problem}.
\begin{definition}\label{eq:DefWeakSol}
A function $u \in \mathbb{X}(\Omega)$ is called a weak solution of \eqref{eq:Problem} if
\begin{equation*}
\begin{aligned}
    \int_{\Omega}\langle \nabla u, \nabla \varphi\rangle\, dx &+\alpha \iint_{\mathbb{R}^{2n}}\dfrac{(u(x)-u(y))(\varphi(x)-\varphi(y))}{|x-y|^{n+2s}}\, dxdy \\
    &= \lambda \int_{\Omega}u\varphi \, dx + \int_{\Omega}f(x,u)\varphi \, dx
    \end{aligned}
\end{equation*}
for every $\varphi \in \mathbb{X}(\Omega)$.
\end{definition}
As usual, weak solutions of \eqref{eq:Problem} can be found as critical points of the functional $\mathcal{J}_{\lambda}:\mathbb{X}(\Omega) \to \mathbb{R}$ defined as
\[
\mathcal{J}_{\lambda}(u) := \dfrac{1}{2}\int_{\Omega}|\nabla u|^2 \, dx + \dfrac{\alpha}{2}\iint_{\mathbb{R}^{2n}} \dfrac{|u(x)-u(y)|^2}{|x-y|^{n+2s}}\, dxdy - \dfrac{\lambda}{2} \int_{\Omega}|u|^2 \, dx - \int_{\Omega}F(x,u)\, dx,
\]
where
\[
F(x,t) := \int_{0}^{t}f(x,\sigma)\, d\sigma, \quad t \in \mathbb{R}.
\]
By assumption $(f_{bc})$ it is standard to prove (see for instance \cite{AmP}) that the functional $\mathcal{J}_{\lambda}$ is Fr\'{e}chet differentiable and that
\begin{align*}
\mathcal{J}_{\lambda}'(u)(\varphi) &= \int_{\Omega}\langle \nabla u, \nabla \varphi\rangle \, dx + \alpha \iint_{\mathbb{R}^{2n}}\dfrac{(u(x)-u(y))(\varphi(x)-\varphi(y))}{|x-y|^{n+2s}}\, dxdy \\
&- \lambda \int_{\Omega}u(x)\varphi(x) \, dx- \int_{\Omega}f(x,u(x))\varphi(x)\, dx \quad \textrm{for every }\varphi \in \mathbb{X}(\Omega)\,.
\end{align*}

Now, we consider the bilinear form 
$\mathcal{B}_{\alpha}:\mathbb{X}(\Omega)\times \mathbb{X}(\Omega) \to \mathbb{R}$, defined by
\begin{equation*}
\mathcal{B}_{\alpha}(u,v):= \int_{\Omega}\langle \nabla u, \nabla v\rangle\, dx +\alpha \iint_{\mathbb{R}^{2n}}\dfrac{(u(x)-u(y))(v(x)-v(y))}{|x-y|^{n+2s}}\, dxdy
\end{equation*}
for any $u,v \in \mathbb{X}(\Omega)$. In spite of the fact that $\alpha$ can be such that $\mathcal{B}_\alpha$ is not positive definite, we give the following definition.

\begin{definition}
We say that $u$ and $v$ are $\mathcal{B}_\alpha$-orthogonal if
\begin{equation*}
\mathcal{B}_{\alpha}(u,v)=0.
\end{equation*}
\end{definition}
The terminology adopted above is justified by the fact that, for $\alpha>0$ (more precisely if $\alpha>-\frac{1}{C_s}$), the bilinear form $\mathcal{B}_{\alpha}$ defines a true scalar product.

We conclude this section dealing with the eigenvalue problem associated to the operator $\LL$, that is the following boundary value problem
\begin{equation}\label{eq.EigenvalueProblem}
\left\{ \begin{array}{rl}
  \LL u = \lambda u, & \textrm{in } \Omega\,,\\
  u= 0, & \textrm{in } \mathbb{R}^n \setminus \Omega\,,
\end{array}\right.
\end{equation}
where $\lambda \in \mathbb{R}$.
According to Definition \ref{eq:DefWeakSol}, we give the following definition.
\begin{definition}
A number $\lambda \in \mathbb{R}$ is called a variational Dirichlet eigenvalue of $\LL$ if there exists a  nontrivial weak solution $u\in \mathbb{X}(\Omega)$ of \eqref{eq.EigenvalueProblem} or, equivalently, if
\begin{equation*}
\int_{\Omega}\langle \nabla u, \nabla \varphi\rangle\, dx +\alpha \iint_{\mathbb{R}^{2n}}\dfrac{(u(x)-u(y))(\varphi(x)-\varphi(y))}{|x-y|^{n+2s}}\, dxdy =\lambda \int_{\Omega}u\varphi \, dx
\end{equation*}
for every $\varphi \in \mathbb{X}(\Omega)$.
If such a function $u$ exists, we call it eigenfunction associated to the eigenvalue $\lambda$.
\end{definition}

Note that the linearity of $\LL$ guarantees a complete description of its eigenvalues, and relative eigenfunctions, according to the following result, see \cite[Proposition 2.4]{MMV}:
\begin{proposition}\label{prop.Eigenvalue}
Let $n>2$. Then the following statements hold true:
\begin{itemize}
\item[(a)] $\LL$ admits a divergent and bounded from below sequence of eigenvalues $\{\lambda_k\}_{k\in \mathbb{N}}$, i.e., there exists $C>0$ such that
\[
-C < \lambda_1 \leq \lambda_2  \leq \ldots \leq \lambda_k \to +\infty\,, \quad  \textrm{as } k \to +\infty.
\]
Moreover, for every $k\in \mathbb{N}$, $\lambda_{k}$ can be characterized as
\begin{equation}\label{eq:lambdaChar}
\lambda_{k} = \min_{\stackrel{u \in \mathbb{P}_{k}}{\|u\|_{L^{2}(\Omega)}=1}} \left\{\int_{\Omega}|\nabla u|^2 \, dx +\alpha \iint_{\mathbb{R}^{2n}}\dfrac{|u(x)-u(y)|^2}{|x-y|^{n+2s}}\, dxdy \right\},
\end{equation}
where 
\begin{equation*}
\mathbb{P}_{1} := \mathbb{X}(\Omega),
\end{equation*}
and, for every $k \geq 2$,
\[
\mathbb{P}_{k} := \left\{ u \in \mathbb{X}(\Omega): \mathcal{B}_{\alpha}(u, u_j)=0 \, \textrm{ for every } j = 1, \ldots,k-1\right\}; 
\]
\item[(b)] for every $k\in \mathbb{N}$ there exists an eigenfunction $u_{k} \in \mathbb{X}(\Omega)$ corresponding to $\lambda_{k}$, which realizes the minimum in \eqref{eq:lambdaChar};
\item[(c)] the sequence $\{u_k\}_{k\in \mathbb{N}}$ of eigenfunctions constitutes an orthonormal basis of $L^{2}(\Omega)$; moreover, the eigenfunctions are $\mathcal{B}_\alpha$-orthogonal.
\item[(d)] for every $k\in \mathbb{N}$, $\lambda_k$ has finite multiplicity.
\end{itemize}
\end{proposition}

\begin{remark}
Clearly, if $\alpha > -\dfrac{1}{C_s}$ there is an improvement on the lower bound of $\lambda_1$, which is thus strictly positive. Moreover, in this case, $\lambda_1$ is also simple.
\end{remark}

We denote by $H_k$ the linear subspace of $\mathbb{X}(\Omega)$ generated by the first $k$ eigenfunctions of $\mathcal{L}_{\alpha}$, i.e.
\[
H_k=\textrm{span}_{\R} \{u_1,\ldots,u_k\}.
\]
Notice that $\mathbb{P}_{k+1}=(H_k)^{\perp_{\mathcal{B}_\alpha}}$, namely the subspace $\mathcal{B}_{\alpha}$-orthogonal to $H_k$.
Also we set
\[
H_k^{0}=\textrm{span}_{\R} \{u_j \ : \  \lambda_j=\lambda_k\},
\]
i.e. the kernel of $\LL-\lambda_k$, and
\[
H_k^{-}=\textrm{span}_{\R} \{u_j \ : \  \lambda_j<\lambda_k\}.
\]
By Proposition \ref{prop.Eigenvalue} (a) we can infer the existence of a positive integer $N_0 \in \mathbb{N}$ such that $\lambda_{N_0}$ is the first (not necessarily simple) positive eigenvalue. Of course, $\lambda_k >0$ for every $k > N_0$.\\
We further notice that, again by Proposition \ref{prop.Eigenvalue},
\begin{equation}\label{eq:Stima1}
\lambda_{k+1} \int_{\Omega}u^2 \, dx \leq \int_{\Omega}|\nabla u|^2 \, dx +\alpha \iint_{\mathbb{R}^{2n}}\dfrac{|u(x)-u(y)|^2}{|x-y|^{n+2s}}\, dxdy
\end{equation}
for every $u \in \textrm{span}(u_1, \ldots, u_k)^{\perp} = \mathbb{P}_{k+1}$ and
\begin{equation}\label{eq:Stima2}
\int_{\Omega}|\nabla u|^2 \, dx +\alpha \iint_{\mathbb{R}^{2n}}\dfrac{|u(x)-u(y)|^2}{|x-y|^{n+2s}}\, dxdy \leq \lambda_k \int_{\Omega}u^2 \, dx
\end{equation}
for every $u \in H_k$.

We first need the following preliminary result inspired by Rabinowitz \cite{Rabinowitz78}, see \cite[Lemma 4.1]{MMV} for a proof.

\begin{lemma}\label{lemma_claim}
Let $k\in \N$ be such that
\[
\lambda_1\leq \lambda_2\leq \ldots \leq \lambda_{k-1}\leq \lambda_k<\lambda_{k+1}\leq\ldots
\]
and decompose the space $\mathbb{X}(\Omega)$ as $\mathbb{X}(\Omega)=H_k \oplus \mathbb{P}_{k+1}$, where  $H_k :=\textrm{span}(u_1, \ldots, u_{k})$.
Then, there exists a positive constant $\beta$ such that for any $u\in \mathbb{P}_{k+1}$
\begin{equation}\label{claim}
    \mathcal{B}_\alpha(u,u)-\lambda_{k}\|u\|_{L^2(\Omega)}^2\geq\beta\|u\|_{\mathbb{X}(\Omega)}^2,
\end{equation}
or, equivalently,
\[
\inf_{u\in \mathbb{P}_{k+1} \setminus\{0\}}\left\{1+\frac{\alpha[u]_s^2-\lambda_{k}\|u\|_{L^2(\Omega)}^2}{\|u\|_{\mathbb{X}(\Omega)}^2}\right\}>0.
\]
\end{lemma}

We now prove a {\it sort of counterpart} of Lemma \ref{lemma_claim} when we restrict our attention to the finite dimensional space $H_k$. The former and the next lemma will be two of the crucial ingredients to verify that the functional $\mathcal{J}_{\lambda}$ verifies the saddle point geometry. 

\begin{lemma}\label{lm:C1upper}
Let $k\in \mathbb{N}$ be such that $\lambda = \lambda_k < \lambda_{k+1}$. Then  there exists a positive constant $\gamma >0$, such that 
\begin{equation}
\label{eq:upperC_1}
\mathcal{B}_{\alpha}(u,u)- \lambda_{k} \|u\|^2_{L^2(\Omega)} \leq - \gamma \|u^-\|_{\mathbb{X}(\Omega)}^2
\end{equation}
for each $u \in H_k$,  where $u=u^0+u^-$, $u^0 \in H_k^0$ and $u^- \in H_k^-$.
\begin{proof}
If $u \equiv 0$, then the assertion is trivial. Hence we assume $u \in H_k \smallsetminus \{0\}$. Thanks to Proposition \ref{prop.Eigenvalue}\, (c), a simple computation yields that
\[
\mathcal{B}_{\alpha}(u,u) - \lambda_{k} \|u\|^2_{L^2(\Omega)} = \mathcal{B}_{\alpha}(u^-,u^-) - \lambda_{k} \|u^-\|^2_{L^2(\Omega)},
\]
which is nonpositive by \eqref{eq:Stima2}. Then it suffices to prove that there exists a positive constant $\gamma>0$ such that 
\begin{equation}
\sup_{u^-\in H_k^- \setminus \{0\}}\left\{1+ \dfrac{\alpha [u^-]^2 -\lambda_k \|u^-\|^2_{L^2(\Omega)}}{\|u^-\|_{\mathbb{X}(\Omega)}^2}\right\} = -\gamma.
\end{equation}
To this aim, we argue as in \cite[Lemma 4.1]{MMV} assuming by contradiction that there exists a sequence $\{u^-_n\}_{n\in \mathbb{N}} \in H_k^- \setminus \{0\}$ such that 
\begin{equation}
1+ \dfrac{\alpha [u_n^-]^2 -\lambda_k \|u_n^-\|^2_{L^2(\Omega)}}{\|u_n^-\|_{\mathbb{X}(\Omega)}^2} \to 0, \quad \textrm{ as } n  \to +\infty.
\end{equation}
We then consider the normalized (in $\mathbb{X}(\Omega)$) sequence $$v_n^- := \dfrac{u_n^-}{\|u_n^-\|_{\mathbb{X}(\Omega)}} \in H^-_k \setminus \{0\},$$
and, since $H^-_k \setminus \{0\}$ is finite dimensional, we can infer the existence of a function $v^- \in H^-_k$ with $\|v^-\|_{\X(\Omega)}=1$ and such that $v_n^- \to v^-$. Therefore,  by the compact embedding of $\X(\Omega)$ in $H^s_0(\Omega)$ and in $L^2(\Omega)$, we find
\begin{equation*}
1+ \alpha [v_n^-]^2 - \lambda_k \|v_n^-\|_{L^2(\Omega)}^2 \to 1+ \alpha [v^-]^2 - \lambda_k \|v^-\|_{L^2(\Omega)}^2, \quad \textrm{as } n \to +\infty,
\end{equation*}
but at the same time, 
\begin{equation*}
1+ \alpha [v_n^-]^2 - \lambda_k \|v_n^-\|_{L^2(\Omega)}^2 \to 0, \quad \textrm{as } n \to +\infty.
\end{equation*}
Since $v^- \in H_k^- \setminus \{0\}$, we also have that
\begin{equation*}
0 > \mathcal{B}_{\alpha}(v^-,v^-) - \lambda_k \|v^-\|^2_{L^2(\Omega)} = \|v^-\|^2_{\mathbb{X}(\Omega)} -1 = 0,
\end{equation*}
and a contradiction arises.
\end{proof}
\end{lemma}

The next Lemma is taken verbatim from \cite{FiSeVaZAMP}.
\begin{lemma}
\label{lm:FboundM}
Let $f$ satisfy $(f_{bc})$. Then there exists a positive constant $\tilde{M}>0$, depending on $\Omega$, such that 
\begin{equation}
\label{eq:estF}
\left | \int_{\Omega} F(x,u(x)) dx \right |\le \tilde{M} \ \|u\|_{\mathbb{X}(\Omega)}
\end{equation}
for all $u \in \mathbb{X}(\Omega)$.
\begin{proof}
By definition of $F$ we have 
\[
\left | \int_{\Omega} F(x,u(x)) dx \right | = \left| \int_{\Omega} \int_0^{u(x)} f(x,t) \, dt \, dx \right| \le M \int_{\Omega} |u(x)| dx 
\]
By the H\"older and Poincar\'e inequalities, we obtain 
\[
M \int_{\Omega} |u(x)| dx \le M |\Omega|^{\frac{1}{2}} \|u \|_{L^2(\Omega)} \le \tilde{M} \|u \|_{\mathbb{X}(\Omega)}.
\]
Hence we get \eqref{eq:estF}, with $\tilde{M}$ depending on $\Omega$.
\end{proof}
\end{lemma}

\section{Proof of Theorem \ref{th:main} and Theorem \ref{th.main2}}
\label{sc:proofs}
The proof of Theorem \ref{th:main} follows the classical streamlines in minimax theory. In particular, and as already mentioned, we will make use of the Saddle Point Theorem by Rabinowitz (see \cite{Rabinowitz78,Rabinowitz86}), and therefore we have to check that its assumptions are satisfied. 

\subsection{Geometry of the functional $\mathcal{J}_{\lambda}$}
 
\begin{lemma}
\label{lm:lowerbound-}
Let $f$ satisfy $(f_{bc})$ and $(F_{-\infty})$ and let $k\in \mathbb{N}$ be such that $ \lambda_k < \lambda_{k+1}$. For every $K>0$, there exists $r=r(K)>0$ such that $\mathcal{J}_{\lambda}(u)\geq K$ for every $u \in \mathbb{P}_{k+1}\oplus H^0_{k}$ with $\|u\|_{\mathbb{X}(\Omega)}\geq r$.
\begin{proof}
Since $u \in \mathbb{P}_{k+1}\oplus H^0_{k}$, we can write $u = u^+ + u^0$, where $u^+ \in \mathbb{P}_{k+1}$ and $u^0 \in H^0_k$.
It now suffices to note that,
\begin{equation}
\begin{aligned}
\mathcal{J}_{\lambda_k}(u) &= \dfrac{1}{2}\mathcal{B}_{\alpha}(u,u) - \dfrac{\lambda_k}{2}\|u\|_{L^2(\Omega)}^2 - \int_{\Omega}F(x,u)\, dx \\
&= \dfrac{1}{2}\mathcal{B}_{\alpha}(u^+,u^+) - \dfrac{\lambda_k}{2}\|u^+\|_{L^2(\Omega)}^2 - \int_{\Omega}F(x,u)\, dx \\
&\geq \dfrac{\beta}{2}\|u^+\|_{\mathbb{X}(\Omega)}^2 - \int_{\Omega}F(x,u)\, dx \quad \mbox{(by Lemma \ref{lemma_claim})}\\
&=\dfrac{\beta}{2}\|u^+\|_{\mathbb{X}(\Omega)}^2 - \int_{\Omega}F(x,u^0)\, dx - \int_{\Omega}\left(F(x,u)-F(x,u^0)\right)\, dx \\
&=\dfrac{\beta}{2}\|u^+\|_{\mathbb{X}(\Omega)}^2 - \int_{\Omega}F(x,u^0)\, dx - \int_{\Omega}\int_{u^0(x)}^{u(x)}f(x,t)dt dx \\
&\geq  \dfrac{\beta}{2}\|u^+\|_{\mathbb{X}(\Omega)}^2 - \int_{\Omega}F(x,u^0)\, dx - \tilde{M}\|u^+\|_{\mathbb{X}(\Omega)},
\end{aligned}
\end{equation}
and the conclusion now easily follows from $(F_{-\infty})$.
\end{proof}
\end{lemma}

A similar statement holds when $(F_{+\infty})$ is in force, namely
\begin{lemma}
\label{lm:lowerbound}
Let $f$ satisfy $(f_{bc})$ and let $k\in \mathbb{N}$ be such that $\lambda_k < \lambda_{k+1}$. For every $K>0$, there exists $r=r(K)>0$ such that $\mathcal{J}_{\lambda}(u)\geq K$ for every $u \in \mathbb{P}_{k+1}$ with $\|u\|_{\mathbb{X}(\Omega)}\geq r$.
\begin{proof}
It suffices to note that, being $u\in \mathbb{P}_{k+1}$,
\begin{equation}
\begin{aligned}
\mathcal{J}_{\lambda_k}(u) &= \dfrac{1}{2}\mathcal{B}_{\alpha}(u,u) - \dfrac{\lambda_k}{2}\|u\|_{L^2(\Omega)}^2 - \int_{\Omega}F(x,u)\, dx \\
&\geq \dfrac{\beta}{2}\|u\|_{\mathbb{X}(\Omega)}^2 - \int_{\Omega}F(x,u)\, dx \quad \mbox{(by Lemma \ref{lemma_claim})}\\
&\geq  \dfrac{\beta}{2}\|u\|_{\mathbb{X}(\Omega)}^2 - \tilde{M}\|u\|_{\mathbb{X}(\Omega)}, \quad \mbox{(by Lemma \ref{lm:FboundM})}.
\end{aligned}
\end{equation}
The conclusion now easily follows.
\end{proof}
\end{lemma}

\begin{remark}
An immediate consequence of Lemma \ref{lm:lowerbound-} or of Lemma \ref{lm:lowerbound} is that
 \begin{equation}\label{eq:liminf}
 \liminf_{ \substack{u \in \mathbb{P}_{k+1} \\ \| u\|_{\mathbb{X}(\Omega)} \to + \infty} } \dfrac{\mathcal{J}_{\lambda_k} (u)}{ \| u\|_{\mathbb{X}(\Omega)}^2} >0.
 \end{equation}
\end{remark}

\begin{proposition}
\label{pr:geometryofSaddlepoint}
Let $f$ satisfy $(f_{bc})$ and let $\lambda_k<\lambda_{k+1}$ for some $k \in \N$. If $(F_{+\infty})$ is in force, then we have 
 \begin{equation}
 \label{eq:limH}
 \lim_{ \substack{u \in H_{k} \\ \| u\|_{\mathbb{X}(\Omega)} \to + \infty} } J_{\lambda_k} (u)= -\infty,
 \end{equation}
 while if $(F_{-\infty})$ holds, then
 \begin{equation}
 \label{eq:limH-}
 \lim_{ \substack{u \in H^{-}_{k} \\ \| u\|_{\mathbb{X}(\Omega)} \to + \infty} } J_{\lambda_k} (u)= -\infty,
 \end{equation}
\end{proposition}
 \begin{proof}
Let us start with the case in which $(F_{+\infty})$ holds.
 Since $u\in H_k$ we can write $u=u^-+u^0$ with $u^- \in H_k^-$ and $u^0 \in H_k^0$. Then we have
 \begin{align*}
 \mathcal{J}_{\lambda_k}(u)=&\dfrac{1}{2}\mathcal{B}_{\alpha}(u,u)-\dfrac{\lambda_k}{2} \int_{\Omega} |u(x)|^2 dx- \int_{\Omega} (F(x,u^-(x)+u^0(x))-F(x,u^0(x))) dx\\
 &\quad -\int_{\Omega} F(x,u^0(x)) dx
 \end{align*}
 Notice that, as in proof of Lemma \ref{lm:FboundM}, we have 
 \begin{align*}
 \left| \int_{\Omega} (F(x,u^-(x)+u^0(x))-F(x,u^0(x))) dx \right| &\le \left|  \int_{\Omega} \int_{u^0(x)}^{u^-(x)+u^0(x)} f(x,t) \, dt dx \right|\\
 &\le M \int_{\Omega} |u^-(x)| dx\\
 &\le \tilde{M} \|u^-\|_{\mathbb{X}(\Omega)}.
 \end{align*}
Thus,  by Lemma \ref{lm:C1upper} and the previous inequality, we obtain 
 \begin{equation}
 \label{eq:Jminf}
 \mathcal{J}_{\lambda_k}(u)\le - \dfrac{\gamma}{2} \|u^-\|_{\mathbb{X}(\Omega)}^2 + \tilde{M} \|u^- \|_{\mathbb{X}(\Omega)} -\int_{\Omega} F(x,u^0(x)) dx.
 \end{equation}
Moreover, by Proposition \ref{prop.Eigenvalue}(c) and the Cauchy-Schwarz inequality, we have 
 \begin{equation*}
 \begin{aligned}
 \|u\|_{\mathbb{X}(\Omega)}^2 &=\|u^0\|_{\mathbb{X}(\Omega)}^2+\|u^-\|_{\mathbb{X}(\Omega)}^2 + 2\langle u^0, u^- \rangle_{\mathbb{X}(\Omega)} \\
 &\leq \|u^0\|_{\mathbb{X}(\Omega)}^2+\|u^-\|_{\mathbb{X}(\Omega)}^2 + 2 \, \|u^0\|_{\mathbb{X}(\Omega)} \, \|u^-\|_{\mathbb{X}(\Omega)}.
 \end{aligned}
 \end{equation*}
 Thus, since $ \|u\|_{\mathbb{X}(\Omega)}$ diverges at $+ \infty$ we have that at least one of the two norms, either $\|u^0\|_{\mathbb{X}(\Omega)}$ or $\|u^-\|_{\mathbb{X}(\Omega)}$, goes to infinity, as well. Assume that $\|u^0\|_{\mathbb{X}(\Omega)} \to + \infty$, then $\|u^-\|_{\mathbb{X}(\Omega)}$ can be finite or infinite. By $(F_{+\infty})$ and by \eqref{eq:Jminf} we get $\mathcal{J}_{\lambda}(u) \to -\infty$. Otherwise, suppose  that $\|u^0\|_{\mathbb{X}(\Omega)}$ is finite, then $\|u^-\|_{\mathbb{X}(\Omega)}$ diverges to $+ \infty$ and by Lemma \ref{lm:FboundM} the last term in \eqref{eq:Jminf} has a linear growth. Hence $\mathcal{J}_{\lambda_k}(u) \to -\infty$. This closes the first part.

The case in which $(F_{-\infty})$ holds is rather simpler. Indeed, keeping the notation $u^-$ for functions in $H^-_{k}$ and reasoning as above, we have that
 \[
 \label{eq:Jminf-}
 \mathcal{J}_{\lambda_k}(u^-)\le - \dfrac{\gamma}{2} \|u^-\|_{\mathbb{X}(\Omega)}^2 + \bar{M} \|u^- \|_{\mathbb{X}(\Omega)},
 \]
an this closes the proof.
\end{proof}

\subsection{Palais-Smale condition}
Let us start recalling the following notion.
\begin{definition}
We say that $\{u_j\}_{j \in \N}$ is a Palais-Smale sequence for $\mathcal{J}_{\lambda}$ at level $c \in \rr$ if
$
\mathcal{J}_{\lambda}(u_j) \to c
$ as $j\to \infty$
and
\begin{equation}
\label{eq:PS}
\mathcal{J}'_{\lambda}(u_j) \to 0, \quad \textrm{ as } j \to +\infty
\end{equation}
holds true.
\end{definition}

\begin{proposition}
\label{pr:ujbounded}
Let $f$ satisfy $(f_{bc})$ and $(F_{\pm\infty})$. Suppose further that $\lambda_k<\lambda_{k+1}$ for some $k \in \N$. If $\{u_j\}_{j \in \N}$ is a Palais-Smale sequence for $\mathcal{J}_{\lambda_k}$, then $\{u_j\}_{j \in \N}$ is bounded in $\mathbb{X}(\Omega)$.
\end{proposition}
\begin{proof}
Let $u_j=u_j^0+u_j^-+u_j^+$ where $u_j^0 \in H_k^0$, $u_j^- \in H_k^-$ and $u_j^+ \in \mathbb{P}_{k+1}$. We will show that all sequence $u_j^0, u_j^-,u_j^+$ are bounded.

Let us start noticing that by \eqref{eq:PS}, we have 
\begin{equation}
\label{eq:ujpm}
\begin{aligned}
\epsilon(1)\| u_j^{\pm}\|_{\mathbb{X}(\Omega)} \ge& \left|\escpr{\mathcal{J}_{\lambda_k} ' (u_j), u_j^{\pm}} \right| \\
 =&\left| \mathcal{B}_{\alpha}(u_j,u_j^{\pm})-\lambda_k \int_{\Omega}  u_j (x) u_j^{\pm}(x)\, dx- \int_{\Omega} f(x,u_j(x)) u_j^{\pm}(x) dx \right|,
\end{aligned}
\end{equation}
where $\epsilon(1)\to 0$ as $j\to \infty$. Since $f$ is bounded, similarly to Lemma \ref{lm:FboundM} we have
\begin{equation}
\label{eq:fujbound}
\left|\int_{\Omega} f(x,u_j(x)) u_j^{\pm} (x) \ dx\right| \le \tilde{M} \| u_j^{\pm}\|_{\mathbb{X}(\Omega)}.
\end{equation}
Thanks to Proposition \ref{prop.Eigenvalue} (c), we have 
\begin{equation}
\label{eq:pm}
\escpr{\mathcal{J}_{\lambda_k} ' (u_j), u_j^{\pm}} = \mathcal{B}_{\alpha}(u_j^{\pm},u_j^{\pm})-\lambda_k \int_{\Omega}   |u_j^{\pm}(x) |^2 \, dx- \int_{\Omega} f(x,u_j(x)) u_j^{\pm}(x) dx.
\end{equation}
Since $u_j^+$ belongs to $\mathbb{P}_{k+1}$, by  Lemma \ref{lemma_claim} and \eqref{eq:fujbound} we get 
\[
\epsilon(1)\| u_j^{+}\|_{\mathbb{X}(\Omega)} \ge \beta \| u_j^+\|_{\X(\Omega)}^2-\tilde{M} \| u_j^{+}\|_{\mathbb{X}(\Omega)},
\]
so that the sequence $\{u_j^+\}_{j\in \N}$ is bounded in $\mathbb{X}(\Omega)$. Furthermore, again by \eqref{eq:ujpm}, \eqref{eq:pm}, Lemma \ref{lm:C1upper}  and \eqref{eq:fujbound} we get
\[
\epsilon(1)\| u_j^{-}\|_{\mathbb{X}(\Omega)} \ge - \escpr{\mathcal{J}_{\lambda_k} ' (u_j), u_j^{-}} \ge \gamma \|u^-\|_{\mathbb{X}(\Omega)}^2- \tilde{M} \|u^-\|_{\mathbb{X}(\Omega)}.
\]
Then the sequence $\{u_j^-\}_{j\in \N}$ is bounded in $\mathbb{X}(\Omega)$, as well. 

We finally prove that $u_j^0$ is bounded in $\mathbb{X}(\Omega)$. First of all, we recall that  $u_j^0$ is an eigenfunction associated to $\lambda_k$, namely
\begin{equation}
\label{eq:u_j0eig}
\mathcal{B}_{\alpha}(u_j^0,u_j^0)=\lambda_{k} \int_{\Omega}  |u_j^0(x)|^2 dx.
\end{equation}
By the Palais-Smale condition, the equation \eqref{eq:u_j0eig} and Proposition \ref{prop.Eigenvalue} (c), we gain 
\begin{align*}
c \leftarrow J_{\lambda_k}(u_j)=&\dfrac{1}{2} \mathcal{B}_{\alpha}(u_j^+,u_j^+)+ \dfrac{1}{2}\mathcal{B}_{\alpha}(u_j^-,u_j^-)-\dfrac{\lambda_k}{2} \int_{\Omega}  \left(|u_j^+(x)|^2+|u_j^-(x)|^2 \right) dx\\
 &- \int_{\Omega} \left(F(x,u_j(x))- F(x,u_j^0(x)) \right)dx-  \int_{\Omega} F(x,u_j^0(x)) dx .
\end{align*}
Thus,  we have that 
\begin{equation}
\label{eq:Ffixed}
\begin{aligned}
& \left|\int_{\Omega} F(x,u_j^0(x)) dx\right| \le |J_{\lambda_k}(u_j)|+ \bigg|\dfrac{1}{2} \mathcal{B}_{\alpha}(u_j^+,u_j^+)+ \dfrac{1}{2}\mathcal{B}_{\alpha}(u_j^-,u_j^-) \\
& \qquad -\dfrac{\lambda_k}{2} \int_{\Omega} \left(|u_j^+(x)|^2+|u_j^-(x)|^2 \right) dx\
 - \int_{\Omega} \left(F(x,u_j(x))- F(x,u_j^0(x)) \right)dx \bigg|
\end{aligned}
\end{equation}
By the Poincaré inequality and the bound on $u_j^+$ and $u_j^-$ we gain 
\[ 
\left|\dfrac{\lambda_k}{2} \int_{\Omega} \left(|u_j^+(x)|^2+|u_j^-(x)|^2 \right) dx\right| \le C\left( \|u_j^+\|_{\mathbb{X}(\Omega)}^2+ \|u_j^-\|_{\mathbb{X}(\Omega)}^2 \right) \le \tilde C
\]
for some $\tilde C>0$ and all $j\in \N$. Moreover,
\begin{align*}
\left|\int_{\Omega} \left(F(x,u_j(x))- F(x,u_j^0(x)) \right)\right| \le & \int_{\Omega}\left| \int_{u_j^0(x)}^{u_j^0(x) + u_j^+(x)+ u_j^-(x)}f(x,t) dt\right|\\
\le & M \int_{\Omega} \left( |u_j^-|+ |u_j^+| \right) dx\\
\le & \tilde{M} \left( \| u_j^-\|_{\mathbb{X}(\Omega)} +\|u_j^+ \|_{\mathbb{X}(\Omega)}  \right) \le C_1
\end{align*}
for some $C_1$ and all $j\in \N$. Therefore, from \eqref{eq:Ffixed}, recalling that $u_j^\pm$ are bounded, we obtain
\[
 \left|\int_{\Omega} F(x,u_j^0(x)) dx\right| \le |c|+o(1)+ \bigg|\dfrac{1}{2} \mathcal{B}_{\alpha}(u_j^+,u_j^+)+ \dfrac{1}{2}\mathcal{B}_{\alpha}(u_j^-,u_j^-) \bigg|+ \tilde{C}+C_1\leq C_2,
\]
where $C_2>0$ is a constant independent of $j$ and $o(1)\to 0$ as $j\to \infty$. Hence the sequence of integrals $\int_{\Omega} F(x,u_j^0(x)) dx$ is bounded. Finally, since $u_j^0$ belongs to $H_k^0$, by $(F_{\pm \infty})$ we get that $u_j^0$ is bounded in $\mathbb{X}(\Omega)$.
\end{proof} 

We establish the validity of the Palais-Smale condition thanks to the following result.
\begin{proposition}
\label{pr:PScondition}
 Let $f$ satisfy $(f_{bc})$ and $(F_{\pm\infty})$. Suppose further that $\lambda_k<\lambda_{k+1}$ for some $k \in \N$.  If $\{u_j\}_{j \in \N}$ is a Palais-Smale sequence for $\mathcal{J}_{\lambda_k}$, then there exists $u_{\infty}$ in $\mathbb{X}(\Omega)$ such that $u_j$ strongly converges to $u_{\infty}$ in $\mathbb{X}(\Omega)$.
\end{proposition}
\begin{proof}
By Proposition \ref{pr:ujbounded} $u_j$ is bounded and $\mathbb{X}(\Omega)$ is reflexive, since $\mathbb{X}(\Omega)$ is an Hilbert space. Then there exists $u_{\infty} \in \mathbb{X}(\Omega)$ such that, up to a subsequence,  $u_j$ weakly converges to $u_{\infty}$ in $\X(\Omega)$. Since $\mathbb{X}(\Omega)$ is compactly embedded in $H^s_0(\Omega)$ (and so in  $L^2(\Omega)$), then, up to a subsequence, $u_j \to u_{\infty}$ in $H^s_0(\Omega)$ (and so in $L^2(\Omega)$) and $u_j\to u_{\infty}$ a.e. in $\Omega$. This implies that
\begin{equation}
\label{eq:Qujinf}
\mathcal{B}_{\alpha}(u_j, \varphi) \to \mathcal{B}_{\alpha}(u_{\infty},\varphi)
\end{equation}
for all $\varphi \in \mathbb{X}(\Omega)$, as $j\to +\infty$.

Since $u_j$ is a Palais-Smale sequence, we have 
\begin{equation}
\label{eq:u_ju_inf}
\begin{aligned}
0 \gets \escpr{ \mathcal{J}_{\lambda_k}' (u_j), u_j-u_{\infty}}=& \mathcal{B}_{\alpha}(u_j,u_j) -\mathcal{B}_{\alpha}(u_{j}, u_{\infty}) \\
&\quad -\lambda_k \int_{\Omega}   u_j(x)(u_j(x)- u_{\infty}(x)) dx  \\
&\quad - \int_{\Omega} f(x,u_j(x)) (u_j (x)- u_{\infty} (x)) dx.
\end{aligned}
\end{equation}
Now, by the H\"older inequality and the bound on $f$ we get 
\begin{align*}
&\left| \lambda_k \int_{\Omega}   u_j(x)(u_j(x)- u_{\infty}(x)) dx + \int_{\Omega} f(x,u_j(x)) (u_j (x)- u_{\infty} (x)) dx\right|\\
&\le \left( \lambda_k \|u_j\|_{L^2(\Omega)}+ M |\Omega|^{\frac{1}{2}} \right) \| u_j-u_{\infty}\|_{L^2(\Omega)} \to 0,
\end{align*}
as $j\to + \infty$. Therefore, passing to the limit in \eqref{eq:u_ju_inf} and taking into account \eqref{eq:Qujinf} we get 
\[
\mathcal{B}_{\alpha}(u_j,u_j) \to \mathcal{B}_{\alpha}(u_{\infty},u_{\infty}).
\]
Since $u_j\to u$ in $H^s_0(\Omega)$, we conclude that  $\|u_j\|_{\mathbb{X}(\Omega)}\to\|u_{\infty}\|_{\mathbb{X}(\Omega)}$. $\X(\Omega)$ being uniformly convex, we conclude that  $u_j \to u_{\infty}$ strongly in $\mathbb{X}(\Omega)$.
\end{proof}

By combining Propositions \ref{pr:ujbounded} and \ref{pr:PScondition} we have the proof of the following compactness property.
\begin{proposition}\label{psc}
 Let $f$ satisfy $(f_{bc})$ and $(F_{\pm\infty})$. Suppose further that $\lambda_k<\lambda_{k+1}$ for some $k \in \N$. Then $\mathcal{J}_{\lambda_k}$ satisfies the Palais-Smale condition at level $c$ for any $c\in \R$, namely every  Palais-Smale sequence at level $c$ admits a strongly convergent subsequence.
\end{proposition}

We are now ready to conclude with the

\begin{proof}[Proof of Theorem $\ref{th:main}$] Let us start fixing some notation. Since $\lambda$ is an eigenvalue, there exists $k\in \N$ such that $\lambda=\lambda_k<\lambda_{k+1}$. Once $k$ has been found, we fix the decomposition
$\mathbb{X}(\Omega)=H_k \oplus\mathbb{P}_{k+1} $, with  $H_k$ having finite dimension.

Let us start with the case in which $(F_{+\infty})$ is in force. From \eqref{eq:liminf} for any $H>0$ there exist $R>0$ such that, if $u \in \mathbb{P}_{k+1} $ and $\|u\|_{\mathbb{X}(\Omega)}\ge R$, then 
\[
J_{\lambda_k}(u)>H.
\]
When $u \in \mathbb{P}_{k+1} $ and $\|u\|_{\mathbb{X}(\Omega)}\le R$, by \eqref{eq:Stima1}, the Rellich-Kondrachov Theorem, the H\"older inequality  and \eqref{eq:Stima1} we have

\begin{align*}
J_{\lambda_k}(u)&\ge \dfrac{\lambda_{k+1}-\lambda_k}{2} \int_{\Omega} |u(x)|^2 dx- \int_{\Omega} F(x,u(x)) dx\\
& \ge - M \int_{\Omega} |u(x)|dx
 \ge  - \tilde{M} \|u(x)\|_{\mathbb{X}(\Omega)}
 \ge - \tilde{M}R=:-C_{R},
\end{align*}

where $\tilde{M}$ is a positive constant. Therefore, we obtain 
\begin{equation}
\label{eq:SPTI4}
J_{\lambda_k}(u) \ge -C_R \quad \text{for any} \quad u \in \mathbb{P}_{k+1}.
\end{equation}

Furthermore, by \eqref{eq:limH} in  Proposition \ref{pr:geometryofSaddlepoint}, there exists $T>0$ such that, for any $u \in H_k$ with $\|u\|_{\mathbb{X}(\Omega)}\geq T$, we have 
\begin{equation}
\label{eq:SPTI3}
J_{\lambda_k}(u)<-C_R.
\end{equation}
Hence, by \eqref{eq:SPTI4} and \eqref{eq:SPTI3} we conclude that
\[
\sup_{\substack{u \in H_{k} \\ \| u\|_{\mathbb{X}(\Omega)} =T}} J_{\lambda_k}(u)< - C_R \le \inf_{u \in \mathbb{P}_{k+1}} J_{\lambda_k}(u),
\]
so the functional $J_{\lambda_k}$ satisfies the geometric assumption $(I_3)$ and $(I_4)$ of \cite[Theorem 4.6]{Rabinowitz86}. Moreover, by Proposition \ref{psc} $J_{\lambda_k}$ satisfies the Palais-Smale condition. Then the Saddle Point Theorem (\cite[Theorem 4.6]{Rabinowitz86}) provides the existence of a critical point $u \in \mathbb{X}(\Omega)$ for the functional $J_{\lambda_k}$ with
\[
\mathcal{J}_{\lambda_k}(u)\leq \max_{v\in H_k \atop \|u\|_{\X(\Omega)}\leq T}\mathcal{J}_{\lambda_k}(v).
\]

\medskip
The case $(F_{-\infty})$ can be treated similarly considering the following decomposition:
$$\mathbb{X}(\Omega) = H^-_{k} \oplus \left(H^0_k\oplus\mathbb{P}_{k+1} \right),$$
where $H^{-}_k$ is the finite dimensional subspace while
$\mathbb{P}_{k+1}\oplus H^0_k$ is the infinite dimensional one. Reasoning as above, by using \eqref{eq:limH-} in place of \eqref{eq:limH} from Proposition \ref{pr:geometryofSaddlepoint}, we conclude the proof of the theorem.
\end{proof}
\medskip

\begin{remark}
Assumption $(f_{bc})$ covers the case $f(x,0)\neq 0$. This implies that the trivial solution is not allowed for this type of nonlinearities, like $f(x,t)=e^{-t^2}\mbox{sign}(t)$.
\end{remark}

Concerning the multiplicity result stated in Theorem \ref{th.main2}, its proof is an easy corollary of Theorem \ref{teosotto} below.

\begin{proof}[Proof of Theorem $\ref{th.main2}$]
We consider first the following decomposition
$$\mathbb{X}(\Omega)= H_k^- \oplus \left( H_k^0 \oplus \mathbb{P}_{k+1}\right).$$
As before, we can assume that $\lambda = \lambda_k < \lambda_{k+1}$ for some $k\in \mathbb{N}$. We now consider the sphere of radius $r>0$ in the finite dimensional subspace $H_k$, namely
$$S := \left\{ u \in H_k: \|u\|_{\mathbb{X}(\Omega)} = r \right\}.$$
By Lemma \ref{lm:C1upper} (since $u \in S \subset H_k$) and $(F_{pos})$, if $r>0$ is small enough, being the norms in $L^\infty(\Omega)$ and in $\X(\Omega)$ equivalent, as $H_k$ is finite dimensional, we get
$$
\sup_{u\in S}J_{\lambda_k}(u) <0.
$$
This fact,  coupled with the lower bound on $J_{\lambda_k}(u)$ for $u\in H_k^0 \oplus \mathbb{P}_{k+1}$ established in \eqref{eq:SPTI4}, allows to apply Theorem \ref{teosotto} with $E= \mathbb{X}(\Omega)$ and $\tilde{E}=H_k^0 \oplus \mathbb{P}_{k+1}$, which yields the desired conclusion, since $\gamma(S)=\mbox{dim}(H_k)$ (see \cite[Remark 5.62]{MMP}).
\end{proof}

\section*{Appendix}
\label{sc:app}
We recall some basic facts about the Krasnoselskii genus and an abstract result due to Rabinowitz.

Let $A\subset \R^N$ be a closed and symmetric set. The genus $\gamma(A)$ of $A$ is defined as the least integer $n$ (if it exists) such that there is an odd function $f\in C(A, \R^n\setminus \{0\})$.

Set $\Sigma: = \{A \subset \R^N\,:\, A \mbox{ is closed and symmetric}\}$.

\begin{theorem}[Theorem 1.9,  \cite{Rabinowitz78}]\label{teosotto}
Let $E$ be a real Banach space and let $I\in  C^1(E,\R)$ be even with $I(0) = 0$ and satisfy the
Palais-Smale condition at any level. Suppose further that
\begin{enumerate}
\item there exists a closed subspace $\tilde E\subset E$ of codimension $j$ and a constant $b$
such that $I_{\vert \tilde E}\geq b$, and
\item there exists $A\in \Sigma$ with $\gamma (A)=m>j$ and $\sup_AI < 0$.
\end{enumerate}
Then $I$ possesses at least $m -j$ distinct pairs of nontrivial critical points.
\end{theorem}


\begin{thebibliography}{100}

\bibitem{ALP}
S. Ahmad, A. C. Lazer, and J. L. Paul, 
{\em Elementary critical point theory and perturbations of
elliptic boundary value problems at resonance}, 
Indiana Univ. Math. J. {\bf 25}(10), (1976), 933--944.

\bibitem{AmP}
A. Ambrosetti, G. Prodi,
{\em A primer of nonlinear analysis}, Cambridge Studies in Advanced Mathematics, 34. Cambridge University Press, Cambridge, (1993).

\bibitem{ArRa}
R. Arora, V. Radulescu,
\emph{Combined effects in mixed local-nonlocal stationary problems}, preprint. \url{https://arxiv.org/abs/2111.06701}

\bibitem{BDVV}
S. Biagi, S. Dipierro, E. Valdinoci, E. Vecchi,
{\em Mixed local and nonlocal elliptic operators: regularity and maximum principles}, 
Comm. Partial Differential Equations {\bf 47} (3) (2022), 585--629.

\bibitem{BDVV2}
S. Biagi, S. Dipierro, E. Valdinoci, E. Vecchi,
{\em A Faber-Krahn inequality for mixed local and nonlocal operators}, J. Anal. Math. \url{https://doi.org/10.1007/s11854-023-0272-5}
 
\bibitem{BDVV3}
S. Biagi, S. Dipierro, E. Valdinoci, E. Vecchi,
{\em A Hong-Krahn-Szeg\"{o} inequality for mixed local and nonlocal operators}, Math. Eng. {\bf 5} (1) (2023), Paper No. 014, 25 pp.
 
\bibitem{BDVV5}
S. Biagi, S. Dipierro, E. Valdinoci, E. Vecchi,
{\em A Brezis-Nirenberg type result for mixed local and nonlocal operators}, preprint. \url{https://arxiv.org/abs/2209.07502}

\bibitem{BMV}
S. Biagi, D. Mugnai, E. Vecchi,
{\em A Brezis-Oswald approach to mixed local and nonlocal operators},  Commun. Contemp.Math. doi:10.1142/S0219199722500572

\bibitem{BMV2}
S. Biagi, D. Mugnai, E. Vecchi,
{\em Necessary condition in a Brezis-Oswald-type problem for mixed local and nonlocal operators}, 
Appl. Math. Lett. {\bf 132}, (2022), 108177.

\bibitem{BMS}
A. Biswas, M. Modasiya, A. Sen,
{\em Boundary regularity of mixed local-nonlocal operators and its application}, Ann. Mat. Pura Appl. {\bf 202}, (2023), 679--710.

\bibitem{Brezis}
H. Brezis,
\emph{Functional analysis, Sobolev spaces and partial differential equations}, 
 U\-ni\-ver\-si\-text, Springer, New York, 2011. 

\bibitem{CDV22}
X. Cabr\'e, S. Dipierro, E. Valdinoci, 
{\em The Bernstein Technique for Integro-Differential
Equations}, Arch. Rational Mech. Anal. {\bf 243}, (2022), 1597--1652.
 
 \bibitem{CKSV}
 Z.-Q. Chen, P. Kim, R. Song, Z. Vondra\v{c}ek, 
 {\em Boundary Harnack principle for $\Delta + \Delta^{\alpha/2}$}, 
 Trans. Amer. Math. Soc. {\bf 364}(8), (2012), 4169--4205. 


 \bibitem{DeFMin}
 C. De Filippis, G. Mingione, 
 {\em Gradient regularity in mixed local and nonlocal problems}, Math. Ann. (2022). \url{https://doi.org/10.1007/s00208-022-02512-7}.

\bibitem{DRV}
E. Di Nezza, G. Palatucci, E. Valdinoci,
{\em Hitchhiker's guide to the fractional Sobolev spaces}, 
Bull. Sci. Math. {\bf 136}, (2012), 521--573.


\bibitem{DPFR}
L.M. Del Pezzo, R. Ferreira, J.D. Rossi, 
{\em Eigenvalues for a combination between local and nonlocal p-Laplacians}
Fract. Calc. Appl. Anal. {\bf 22}(5), (2019), 1414--1436.

\bibitem{DPLV1}
S. Dipierro, E. Proietti Lippi, E. Valdinoci, 
{\em Linear theory for a mixed operator with Neumann conditions}, Asymptotic Analysis {\bf 128}(4), (2022), 571--594.

\bibitem{DPLV2}
S. Dipierro, E. Proietti Lippi, E. Valdinoci, 
{\em (Non)local logistic equations with Neumann conditions}, Ann. Inst. H. Poincaré Anal. Non Linéaire (2022), DOI: 10.4171/AIHPC/57 .

\bibitem{DV}
S. Dipierro, E. Valdinoci,
{\em Description of an ecological niche for a mixed local/nonlocal dispersal: an evolution equation and a new Neumann condition arising from the superposition of Brownian and L\'{e}vy processes},
Phys. A. {\bf 575}, (2021), 126052.


\bibitem{FiSeVaZAMP}
A. Fiscella, R. Servadei, E. Valdinoci, 
{\em A resonance problem for non-local elliptic operators},
Z. Anal. Anwend. {\bf 32}(4), (2013), 411--431.


\bibitem{FoGa}
A. Fonda, M. Garrione, {\em Nonlinear resonance: a comparison between Landesman-Lazer and Ahmad-Lazer-Paul conditions}, Adv. Nonlinear Studies {\bf 11}, (2011), 391--404.


\bibitem{Garain}
P. Garain,
{\em On a class of mixed local and nonlocal semilinear elliptic equation with singular nonlinearity}, J. Geom. Anal. {\bf 33}, 212, (2023). 


\bibitem{GarainKinnunen}
P. Garain, J. Kinnunen,
{\em On the regularity theory for mixed local and nonlocal quasilinear elliptic equations}, Trans. Amer. Math. Soc. {\bf 375}(8), (2022), 5393--5423.
 
\bibitem{GarainLindgren}
P. Garain, E. Lindgren, 
{\em Higher Hölder regularity for mixed local and nonlocal degenerate elliptic equations}, Calc. Var. {\bf 62}, 67, (2023).
 
\bibitem{GarainUkhlov}
P. Garain, A. Ukhlov,
{\em Mixed local and nonlocal Sobolev inequalities with extremal and associated quasilinear singular elliptic problems}, Nonlinear Anal. {\bf 223}, (2022), 113022.

\bibitem{Ibdah}
H. Ibdah, 
{\em Strong solutions to a modified Michelson-Sivashinsky equation}, 
Commun. Math. Sci.  {\bf 19}(4), (2021), 10711100.


\bibitem{lions}
J.-L. Lions, E. Magenes, {\em Non-Homogeneous Boundary Value Problems and Applications. Vol. I}, Springer-Verlag 1972.


\bibitem{MMV}
A. Maione, D. Mugnai, E. Vecchi,
{\em Variational methods for nonpositive mixed local--nonlocal operators}, Fract. Calc. Appl. Anal. \url{https://doi.org/10.1007/s13540-023-00147-2}

\bibitem{LIBRO}
G. Molica Bisci, V. Radulescu, R. Servadei,
{\em  Variational methods for nonlocal fractional problems},
Encyclopedia of Mathematics and its Applications, 162. Cambridge University Press, Cambridge, 2016. xvi+383 pp.

\bibitem{MMP}
D. Motreanu, V.V. Motreanu, N.S. Papageorgiou, Topological and Variational Methods with Applications to Nonlinear Boundary Value
Problems, Springer. 2012.


\bibitem{MPL}
D. Mugnai, E. Proietti Lippi, 
{\em On mixed local-nonlocal operators with $(\alpha , \beta )-$Neumann conditions},   
Rend. Circ. Mat. Palermo (2) (2022). 
doi:10.1007/s12215-022-00755-6

\bibitem{Rabinowitz78}
P.H. Rabinowitz, 
{\em Some minimax theorems and applications to nonlinear partial differential equations}, 
In Nonlinear analysis (collection of papers in honor of Erich H. Rothe), (1978), 161--177.

\bibitem{Rabinowitz86}
P.H. Rabinowitz, 
Minimax methods in critical point theory with 
applications to differential equations, volume 65 of CBMS Regional Conference Series in Mathematics. Published for
the Conference Board of the Mathematical Sciences, Washington, DC; by the American Mathematical Society, Providence, RI, 1986.

\bibitem{SalortVecchi}
A.M. Salort, E. Vecchi,
{\em On the mixed local--nonlocal H\'{e}non equation},
Differential Integral Equations {\bf 35}(11-12), (2022).


\bibitem{Siva}
G.I. Sivashinsky, 
{\em Nonlinear analysis of hydrodynamic instability in laminar 
ames- I. Derivation of basic equations}, 
in P. Pelc\'{e} (ed.), Dynamics of Curved Fronts, Academic Press, San Diego,
459--488, 1988.

\bibitem{SUPR} 
X. Su, E. Valdinoci, Y. Wei, J.Zhang,
{\em Regularity results for solutions of mixed local and nonlocal elliptic equations}, Math. Z. {\bf 302}, (2022), 1855--1878.



\end{thebibliography}
\end{document}